\documentclass[12pt,leqno]{amsart}
\usepackage{amscd}
\usepackage{color}
\usepackage{amssymb}
\usepackage{amsfonts}
\usepackage{latexsym}
\usepackage{verbatim}

\theoremstyle{plain}
\newtheorem{theorem}{Theorem}[section]
\newtheorem{conj}[theorem]{Conjecture}

\newtheorem{cor}[theorem]{Corollary}

\renewcommand{\b}{\begin{equation}}
\newcommand{\e}{\end{equation}}

\title[]{A note on the pluriclosed flow on balanced manifolds with $c_1=0$}

\makeatletter
\@namedef{subjclassname@2020}{%
 \textup{2020} Mathematics Subject Classification}
\makeatother


\address{ (Anna Fino) Dipartimento di Matematica ``G. Peano''\\
Universit\`{a} degli studi di Torino \\
Via Carlo Alberto 10\\
10123 Torino, Italy\\
\& Department of Mathematics and Statistics\\
Florida International University\\
Miami, FL 33199, United States}
\email{annamaria.fino@unito.it, afino@fiu.edu}

\address{(Luigi Vezzoni) Dipartimento di Matematica G. Peano \\ Universit\`a di Torino\\
Via Carlo Alberto 10\\
10123 Torino\\ Italy}
\email{luigi.vezzoni@unito.it}

\begin{document}
\author{Anna Fino and Luigi Vezzoni}
\maketitle

\date{\today}
\begin{abstract}
%We conjecture that on a compact balanced manifold $M$ with $c_1(M)=0$ %the pluriclosed flow has a long-time solutions $\omega_t$ for every %initial pluriclosed metric and that $\omega_t$ converges smoothly to %a K\"ahler metric. We observe that this happens when $M$ is a compact %quotient of a Lie group, $\omega_B$ is left-invariant with vanishing %Chern-Ricci form and $\omega_t$ is left-invariant. This in particular %gives a nee confirmations to the Fino-Vezzoni conjecture.  
We conjecture that on any compact balanced manifold  $(M, \omega_B)$  with  $c_{1}(M)=0$, the pluriclosed flow admits long-time solutions  $\omega_{t}$ for every initial pluriclosed metric, and that  $\omega_{t}$ converges smoothly to a K\"ahler metric as  $t \to \infty$.  We verify that this phenomenon occurs when $M$ is a compact quotient of a Lie group by a discrete subgroup, the background metric $\omega_{B}$ is invariant with vanishing Chern--Ricci form, and the initial metric $\omega_{0}$ is invariant. In particular, this provides new evidences for the Fino-Vezzoni conjecture.\end{abstract}

\section{Introduction}
A Hermitian metric on a complex manifold $(M,J)$ is a Riemannian metric $g$ which is compatible with the complex structure $J$, in the sense that
$g(JX, JY) = g(X,Y)$
for all  vector fields $X,Y$. The associated fundamental $(1,1)$-form (or Hermitian form) $\omega$ is defined by $\omega(X,Y) = g(JX, Y).$ From now on, we will  identify the Hermitian metric $g$ with its associated fundamental form $\omega$.
Recall that a Hermitian metric $\omega$ on a compact complex manifold of complex dimension $n$  is 
\emph{balanced} if 
$$d\omega^{n-1}=0,$$ where  $\omega^{n-1}$ denotes the $(n-1)$-fold wedge product, and \emph{pluriclosed} (or SKT) if 
$$
\partial\bar\partial \omega = 0.$$ A natural question is to study the interplay between the balanced and pluriclosed conditions. A Hermitian metric which is balanced  and  pluriclosed is necessarily  K\"ahler  \cite{AI, Popovici}.
In \cite{FV} the authors proposed the following still open conjecture
\begin{conj}\label{finovezzoni}
If a compact complex manifold $M$ admits a balanced
metric and a pluriclosed metric, then it must admit a K\"ahler metric.
\end{conj}

The conjecture was mainly motivated by the papers \cite{chiose,FuLiYau,verb} where it is proved the non-existence of pluriclosed metrics on some compact non-K\"ahler complex  manifolds admitting balanced metrics. 
Some partial confirmations of the conjecture are given in \cite{CZ, FGV25,FP,finoparadiso,FV,FV2,GZ, GiPo, Kwong,LiZheng, Otiman,podesta,zheng}. Non-compact counterexamples of this conjecture have been instead constructed in \cite{_Freibert-Swann_}.

\medskip 
In the present paper we focus on Conjecture \ref{finovezzoni} when the first Chern class $c_1(M)$ of the complex manifold $M$ is vanishing. 
In this case the pluriclosed flow \cite{ST} appears as a natural tool for studying Conjecture \ref{finovezzoni} since, accordingly to a conjecture in \cite{ST}, condition $c_1(M)=0$ should imply that solutions to the pluriclosed flow exist for every positive time and we expect that a background balanced metric could be used for studying the long-time behaviour of the flow.

\medskip 
We state, accordingly, the following new conjecture

\begin{conj}\label{newfinovezzoni}   
Let $(M, \omega_B)$ be a compact complex manifold with vanishing first Chern class, equipped with a balanced metric. Assume that $M$ admits a pluriclosed metric $\omega_0$. Then the pluriclosed flow with initial condition $\omega_{t\mid t=0} = \omega_0$ exists for all time and converges smoothly to a K\"ahler metric. In particular, $M$ is K\"ahler.
\end{conj}

As evidence for Conjecture~\ref{newfinovezzoni}, we prove the following theorem. 

\begin{theorem}\label{main}
Let $M=G/\Gamma$ be a compact quotient of a Lie group by a  lattice $\Gamma$, endowed with an invariant complex structure. Assume that $M$ has an invariant balanced metric $\omega_B$ such that $\rho^{c}(\omega_B)\leq 0$. If $M$ admits an invariant pluriclosed metric $\omega_0$, then the pluriclosed flow with initial condition $\omega_{t|t=0}=\omega_0$ has a long-time solution which converges smoothly to a K\"ahler metric. 
\end{theorem}

Here and in what follows, by an invariant geometric structure on $G/\Gamma$ we mean a structure induced by a left-invariant tensor on $G$, i.e. one that descends from a tensor on $G$ invariant under left translations. Moreover, we recall that on a Hermitian manifold $(M,\omega)$ the {\em Chern Ricci-form} $\rho^c$ is locally defined by 
$$
\rho^c=-i\partial \bar \partial \log\det \omega. 
$$
The form $\rho^c$ is closed and its cohomology class is $2\pi c_1(M)$, where $c_1(M)$ is the first Chern class of $M$.
In light of \cite{V}, the condition $\rho^c \leq 0$ is in fact a condition on the Lie group $G$, since $\rho^c$ is independent of the choice of left-invariant Hermitian metric used in its computation. Moreover, by Michelsohn’s criterion \cite{M}, the condition $\rho^c(\omega_B) \leq 0$ on a compact balanced manifold $M$ with $c_1(M)=0$ forces $\rho^c(\omega_B)=0$, since any exact semi-negative definite $(1,1)$-form must vanish. Hence Theorem~\ref{main} is relevant only in the case $\rho^c=0$, although we prefer to state it in its full generality.

\medskip

As a consequence of Theorem~\ref{main}, we confirm Conjecture~\ref{finovezzoni} for compact quotients of Lie groups endowed with an invariant complex structure, under the assumption that invariant Hermitian metrics have semi-negative definite Chern--Ricci form.

\begin{cor}\label{cor}
Let $M = G/\Gamma$ be a compact quotient of a Lie group by a lattice, endowed with an invariant complex structure. Assume that the Chern--Ricci form of left-invariant metrics on $G$ is semi-negative definite. If $M$ admits both a pluriclosed and a balanced metric, then it admits a K\"ahler metric.
\end{cor}

The corollary in particular applies when $G$ is nilpotent or almost abelian, since in both cases left-invariant metrics satisfy $\rho^c \leq 0$. Thus, \cite[Theorem 1.1]{FV2} and \cite[Theorem 3.6]{finoparadiso} can be recovered as special cases of Corollary~\ref{cor}. 
Also the well-known fact that compact complex parallelizable manifolds do not admit pluriclosed metrics (see e.g. \cite{brienzafino}) is implied by Corollary \ref{cor}. 

Other cases where the result applies are provided, for instance, by solvmanifolds $G/\Gamma$ whose Lie algebra has a $J$-invariant nilradical of codimension $2$, since, by \cite[Theorem 2.1]{brienzafino}, under these assumptions left-invariant metrics are Chern-flat.

To the best of our knowledge, this provides new instances where the conjecture is verified. Moreover, it extends previously known cases in which the nilradical is abelian \cite{CZ,GZ}.

%Finally, Corollary~\ref{cor} also applies in the %setting of \cite[Theorem 5.1]{GiPo}. Indeed, if %$(M,J,h)$ is a balanced Hermitian manifold with %$M=G_0/\Gamma$, where $G_0$ is a non-compact real %simple Lie group of inner type, endowed with a %standard invariant complex structure $J$ and an %invariant balanced metric $h$, then by \cite[Theorem %5.1]{GiPo} one has $c_1(M)=0$ and the metric $h$ has %vanishing Chern scalar curvature \textcolor{red}{Qui %sono confuso perch\'e noi necessitiamo che tutto la %forma di Chern-Ricci sia nulla, mentre se capisco %bene nel caso GiPo solo la curvatura scalare si %annulla, hai ragione!}. In particular, all the %hypotheses of Corollary~\ref{cor} are satisfied, and %we recover that the conjecture holds in this case, in %agreement with the result previously proved in &\cite{GiPo}.

\medskip

\noindent \textbf{Acknowledgments:} The authors would like to thank Elia Fusi and Giovanni Gentili for usefull observations.  
The authors are partially supported by GNSAGA (INdAM). Anna Fino is also supported by Project PRIN 2022 \lq \lq Real and complex manifolds: Geometry and Holomorphic Dynamics” and  by  a grant from the Simons Foundation (\#944448).

\section{Proof of Theorem \ref{main}}
In this section we focus on the proof of Theorem~\ref{main}. Let $M = G/\Gamma$ be a compact quotient of a Lie group $G$ by a lattice $\Gamma$, endowed with an invariant complex structure. Assume that $M$ admits an invariant balanced metric $\omega_B$ such that $\rho^c(\omega_B) \leq 0$.

 Consider the pluriclosed flow $\omega_t$ with initial datum $\omega_0$, where $\omega_0$ is an invariant  pluriclosed metric on $M$.

We recall that $\omega_t$ solves 
$$
\partial_t\omega_t=-\rho^b(\omega_t)^{1,1}\,,
$$
where $\rho^b(\omega_t)$ is the Bismut-Ricci form of $\omega_t$ and the upper script $1,1$ denotes the $(1,1)$-component.  
Since $\omega_0$ is invariant, $\omega_t$ remains invariant for all $t$.   

Moreover, since $c_1(M)=0$, the Bismut--Ricci form $\rho^b(\omega_t)$ is exact for every $t$. Hence,
\[
\frac{d}{dt} \int_M \omega_t \wedge \omega_B^{n-1}
= - \int_M \rho^b(\omega_t) \wedge \omega_B^{n-1} = 0,
\]
where we used that $\omega_B$ is balanced, so $d(\omega_B^{n-1})=0$, and thus the integral of an exact form against $\omega_B^{n-1}$ vanishes by Stokes' theorem.

Since
\[
\omega_t \wedge {\omega_B}^{n-1} = \mathrm{tr}_{{\omega_B}}(\omega_t)\,{\omega_B}^n,
\]
and since $\omega_t$ and $\omega_B$ are invariant, the function $\mathrm{tr}_{\omega_B}(\omega_t)$ is invariant on $M$ and therefore constant.

Moreover, from \cite[Lemma 6.1]{S2} we have
\[
\frac{d}{dt} \log \frac{{\omega_B}^n}{\omega_t^n}
= -|T|_{\omega_t}^2 + \mathrm{tr}_{\omega_t} \rho^{c}({\omega_B}),
\]
where $T$ denotes the torsion tensor of the Chern  connection associated to the Hermitian metric $\omega_t$, and $|T|_{\omega_t}^2$ is its pointwise norm with respect to $\omega_t$.

The assumption  $\rho^{c}(\omega_B)\leq 0$ then implies 
$$
\frac{d}{dt}\log\frac{ \omega_B^n}{\omega_t^n}\leq 0
$$
and  hence that $\log\frac{\omega_B^n}{\omega_t^n}$ is bounded from above along the flow. It follows 
$$
\frac{1}{C}\omega_B \leq \omega_t \leq C \omega_B 
$$
for a uniform constant $C$ and so by applying \cite[Theorems 1.7 and 1.8]{S2} we have that $\omega_t$ is defined in $[0,\infty)$ and its norm is uniformly bounded. We then  argue as in the proof of \cite[Theorem 1.1]{S2} (see also \cite[Section 9.8]{mariobook})   to conclude that every sequence $t_j \to \infty$ admits a subsequence $t_{j_k}$ such that $\omega_{t_{j_k}}$ converges smoothly to a steady soliton $\omega_{\infty}$ of the pluriclosed flow.

Assumption $\rho^{c}(\omega_B)\leq 0$ implies that $\omega_{\infty}$ is K\"ahler Ricci-flat. This follows from Lemma~6.3 in \cite{S2}, which gives a rigidity result for steady pluriclosed solitons under this curvature assumption. Hence $\omega_{\infty}$ is K\"ahler Ricci-flat, and the result follows. \hfill $\square$

\smallskip

 We now deduce a direct consequence of Theorem~\ref{main}.

\begin{cor}
Let $M=G/\Gamma$ be a compact quotient of a Lie group by a lattice, endowed with an invariant complex structure. Assume that $M$ admits an invariant balanced metric $\omega_B$ such that $\rho^{c}(\omega_B)\leq 0$. If $M$ admits a pluriclosed metric, then $M$ admits a K\"ahler metric.
\end{cor}

\begin{proof}
It is enough to observe that if $M$ admits a pluriclosed metric, then it also admits an invariant pluriclosed metric, obtained by the symmetrization process on the Lie group $G$. This procedure consists in averaging the metric over the compact quotient $M=G/\Gamma$, and it preserves the pluriclosed condition (see \cite{Belgun,Fino-Grantcharov, Ugarte}). Hence all the assumptions of Theorem~\ref{main} are satisfied, and the conclusion follows.
\end{proof}

\end{document}